\documentclass{article}
\usepackage[utf8]{inputenc}

\usepackage{amsmath,amsfonts,amssymb,amsthm,enumerate,hyperref,ytableau,authblk,listings}
\usepackage[capitalize]{cleveref}
\usepackage{microtype}
\usepackage{fullpage,color}
\newtheorem{theorem}{Theorem}[section]
\newtheorem{lemma}[theorem]{Lemma}

\newtheorem{conjecture}[theorem]{Conjecture}

\theoremstyle{definition}

\providecommand{\coset}[2]{[#1 \mapsto #2]}
\providecommand{\dcoset}[2]{[#1 \leftrightarrow #2]}
\providecommand{\id}{\mathsf{id}}
\providecommand{\cM}{\mathcal{M}}

\title{Uniqueness for 2-Intersecting Families of \\ Permutations and Perfect Matchings}
\author{Gilad Chase}
\author{Neta Dafni}
\author{Yuval Filmus}
\author{Nathan Lindzey}
\affil{Technion --- Israel Institute of Technology}

\begin{document}

\maketitle

\begin{abstract}
    We give a characterization of the largest $2$-intersecting families of permutations of $\{1,2,\ldots,n\}$ and of perfect matchings of the complete graph $K_{2n}$ for all $n \geq 2$.
\end{abstract}

\section{Introduction} \label{sec:introduction}

Erd\H{o}s--Ko--Rado theory~\cite{EKR} studies intersecting families of objects. One of the high points of the theory is the Ahlswede--Khachatrian theorem~\cite{AK1,AK2}.

\begin{theorem} \label{thm:AK}
Let $\mathcal{F}$ be a subset of $\binom{[n]}{k}$ (the collection of all subsets of $[n] := \{1,\ldots,n\}$ of size $k$) which is \emph{$t$-intersecting}: every $A,B \in \mathcal{F}$ satisfy $|A \cap B| \geq t$. Then
\[
 |\mathcal{F}| \leq \max_{r \leq (n-t)/2} \left|\left\{ A \subseteq \binom{[n]}{k} : |A \cap [t+2r]| \geq t+r \right\}\right|.
\]
Furthermore, if $\mathcal{F}$ achieves this bound, then
\[
 \mathcal{F} = \left\{ A \subseteq \binom{[n]}{k} : |A \cap S| \geq t+r \right\}
\]
for some $r \le (n-t)/2$ and some $S \subseteq [n]$ of size $t + 2r$.
\end{theorem}

The theorem consists of two statements: an \emph{upper bound} on the size of $t$-intersecting families, and a characterization of the extremal families. The latter part is known as \emph{uniqueness}.

The Ahlswede--Khachatrian theorem is about intersecting families of sets.
In this paper, we will be interested in intersecting families of permutations and perfect matchings. 

\paragraph{Permutations}
Let $S_n$ be the group of all permutations of $[n]$. Two permutations $\alpha,\beta \in S_n$ are \emph{$t$-intersecting} if there exist $t$ distinct indices $i_1,\ldots,i_t \in [n]$ such that $\alpha(i_1) = \beta(i_1),\ldots,\alpha(i_t) = \beta(i_t)$. Ellis, Friedgut and Pilpel~\cite{EFP} conjectured that the Ahlswede--Khachatrian theorem extends to intersecting families of permutations.

\begin{conjecture} \label{cnj:AK-sym}
Let $\mathcal{F}$ be a subset of $S_n$ which is $t$-intersecting. Then
\[
 |\mathcal{F}| \leq \max_{r \leq (n-t)/2} \left|\left\{ \alpha \in S_n : \alpha(i) = i \text{ for at least } t+r \text{ many } i \in [t+2r] \right\}\right|.
\]
Furthermore, if $\mathcal{F}$ achieves this bound, then
\[
 \mathcal{F} = \left\{ \alpha \in S_n : \alpha(i_s) = j_s \text{ for at least } t+r \text{ many } s \in [t+2r] \right\}
\]
for some $r \leq (n-t)/2$ and some distinct $i_1,\ldots,i_{t+2r} \in [n]$ and distinct $j_1,\ldots,j_{t+2r} \in [n]$.
\end{conjecture}

When $t \leq 3$, the maximum is always attained at $r = 0$, in which case the upper bound is $(n-t)!$, and the conjectured extremal families are the \emph{$t$-cosets}.

Deza and Frankl~\cite{DezaFrankl} proved the upper bound part of \Cref{cnj:AK-sym} in the special case $t = 1$ by noticing that the $n$ cyclic rotations of any fixed permutation are pairwise non-$1$-intersecting. Proving uniqueness in this case proved a lot harder, but eventually many proofs were found~\cite{CameronKu,LaroseMalvenuto,GodsilMeagher,EFP}.

The case $t > 1$ is significantly harder. For every $t > 1$, Ellis, Friedgut and Pilpel~\cite{EFP} proved that the upper bound in \Cref{cnj:AK-sym} holds for large enough $n$, and Ellis~\cite{EllisCameronKu} proved that the uniqueness part of \Cref{cnj:AK-sym} holds for large enough $n$. Ellis, Friedgut and Pilpel used a spectral technique based on the so-called Hoffman bound (closely related to the Lov\'asz theta function and to the linear programming bound in coding theory), which has found many other applications in Erd\H{o}s--Ko--Rado theory~\cite{GodsilMeagherBook}.

Recently, Meagher and Razafimahatratra~\cite{MR21}, refining the techniques of Ellis, Friedgut and Pilpel, proved that $2$-intersecting subsets of $S_n$ contain at most $(n-2)!$ permutations, thus verifying the upper bound part of \Cref{cnj:AK-sym} for $t = 2$ and \emph{all} $n$. However, they were unable to characterize the extremal families in this case. In this paper, we show that all such families are $2$-cosets, thus verifying the uniqueness part of \Cref{cnj:AK-sym} for $t = 2$ and all $n$.

\begin{theorem} \label{thm:main-sym}
Let $n \geq 2$. If $\mathcal{F}$ is a $2$-intersecting subset of $S_n$ of size $(n-2)!$, then there exist $i_1 \neq i_2$ and $j_1 \neq j_2$ such that
\[
 \mathcal{F} = \{ \alpha \in S_n : \alpha(i_1) = j_1 \text{ and } \alpha(i_2) = j_2 \}.
\]
\end{theorem}

\paragraph{Perfect matchings} So far we have considered intersecting families of permutations. A related area of study is intersecting families of perfect matchings of the complete graph $K_{2n}$. This can be seen as the non-bipartite analog of the symmetric group, which is the set of perfect matchings in the complete bipartite graph $K_{n,n}$.

A simple argument shows that a $1$-intersecting family of perfect matchings contains at most $(2n-3)!! = (2n-3)(2n-5)\cdots(1)$ perfect matchings. More difficult arguments~\cite{GodsilMeagherPM,Lindzey17,MeagherMoura,DFLLV} show that this is attained uniquely by $1$-cosets. 

Lindzey~\cite{Lindzey18,LindzeyThesis} extended the arguments of Ellis, Friedgut and Pilpel~\cite{EFP,EllisCameronKu} to the setting of perfect matchings, proving that for all $t$, the maximum size $t$-intersecting families are $t$-cosets, for large enough $n$.

Recently, Fallat, Meagher and Shirazi showed that the maximum size of a $2$-intersecting family is at most $(2n-5)!!$ for all $n$. Extending our arguments in the setting of the symmetric group, we prove an analog of \Cref{thm:main-sym} in the setting of the perfect matching scheme.

\begin{theorem} \label{thm:main-pms}
Let $n \geq 2$. If $\mathcal{F}$ is a $2$-intersecting subset of $\cM_{2n}$ of size $(2n-5)!!$, then there exist distinct $i_1,j_1,i_2,j_2$ such that $\mathcal{F}$ consists of all perfect matchings containing the edges $\{i_1,j_1\},\{i_2,j_2\}$.
\end{theorem}

\paragraph{On the proof} The spectral technique used by Ellis, Friedgut and Pilpel shows that for every $t \geq 1$, if $n$ is large enough and $\mathcal{F} \subseteq S_n$ is a $t$-intersecting family of size $(n-t)!$, then the characteristic function of $\mathcal{F}$, denoted as $1_{\mathcal{F}}$, has \emph{degree} at most~$t$. This means that $1_{\mathcal{F}}$ can be expressed as a polynomial of degree at most~$t$ in the Boolean variables $x_{ij}$, where $x_{ij} = 1$ if the input permutation maps $i$ to $j$. Equivalently, $x_{ij}$ is the $(i,j)$'th entry of the permutation matrix representing the input permutation.

Similarly, Meagher and Razafimahatratra show that if $n \geq 5$ and $\mathcal{F} \subseteq S_n$ is a $2$-intersecting family of size $(n-2)!$, then the characteristic function of $\mathcal{F}$ has degree at most~$2$. This is the starting point of this work.

Ellis, Friedgut and Pilpel used polyhedral techniques (chiefly, the Birkhoff--von~Neumann theorem on bistochastic matrices) to show that Boolean degree~$1$ functions are \emph{dictators}, that is, they either depend only on $\alpha(i)$ for some $i \in [n]$ (where $\alpha$ is the input permutation), or they depend only on $\alpha^{-1}(j)$ for some $j \in [n]$. The only $1$-intersecting dictators are $1$-cosets, and this provides a proof of the uniqueness part of \Cref{cnj:AK-sym} when $t = 1$.

Similar techniques do not work for larger degree~\cite{FilmusComment}.
Instead, we turn to the theory of \emph{complexity measures of Boolean functions}~\cite{BuhrmandeWolf}, which was recently generalized to the symmetric group and to the perfect matching scheme by Dafni et al.~\cite{DFLLV}. Using these techniques, Dafni et al.\ give a simple proof of the uniqueness part of \Cref{cnj:AK-sym}, as well as of its analog for the perfect matching scheme, for large $n$.

The chief tool used by Dafni et al., which we also employ in our proof, is \emph{certificate complexity}. Let $f\colon S_n \to \{0,1\}$ be a Boolean function, and let $\alpha \in S_n$ be such that $f(\alpha) = b$. A \emph{certificate} for $\alpha$ is a subset $\{i_1,\ldots,i_m\} \subseteq [n]$ such that $f(\beta) = b$ whenever $\beta(i_1) = \alpha(i_1),\ldots,\beta(i_m) = \alpha(i_m)$. The idea is that in order to verify that $f(\alpha) = b$, it suffices to check the value of $\alpha(i_1),\ldots,\alpha(i_m)$. The \emph{certificate complexity} of $\alpha$ is the minimum size of a certificate for $\alpha$, and the certificate complexity of $f$ is the maximum, over all $\alpha \in S_n$, of the certificate complexity of $\alpha$. We denote the certificate complexity of $f$ by $C(f)$.

Suppose that $f$ is the characteristic function of a $2$-intersecting subset of $S_n$. When $C(f) \leq n-2$, we give a structure theorem for $f$ which suffices to bound the size of the family away from $(n-2)!$ unless $C(f) = 2$, in which case $f$ is the characteristic function of a $2$-coset. A simple argument (using another complexity measure, \emph{sensitivity}) shows that if $n \geq 8$ and $f\colon S_n \to \{0,1\}$ has degree at most~$2$ then $C(f) \leq n-2$, allowing us to apply the preceding argument. Finally, we handle the case $n \le 7$ using exhaustive search, employing an algorithm for enumerating maximum cliques.

Similar techniques work in the case of the perfect matching scheme, with one complication: the proof of our structure theorem relies on the fact that Boolean degree~$1$ functions on the symmetric group have certificate complexity~$1$, but this fails for the perfect matching scheme. Using the maximality of maximum-size $2$-intersecting families, we are able to overcome this hurdle.

\paragraph{Future research} \Cref{cnj:AK-sym} is about families of permutations in which any two permutations agree on the images of $t$ points. A related question concerns $t$-setwise-intersecting families of permutations, in which any two permutations $\alpha,\beta$ agree on the image of a set of size $t$: $\{\alpha(i_1),\ldots,\alpha(i_t)\} = \{\beta(i_1),\ldots,\beta(i_t)\}$ for a set of size $t$. 

Ellis~\cite{EllisSetwise} showed that for every $t$, if $n$ is large enough then the maximum size of a $t$-setwise-intersecting subset of $S_n$ is $t! (n-t)!$, and this is achieved uniquely by families of the form
\[
 \bigl\{\alpha \in S_n : \{\alpha(i_1),\ldots,\alpha(i_t)\} = \{j_1,\ldots,j_t\}\bigr\},
\]
which we call \emph{$t$-setwise-cosets}.

Meagher and Razafimahatrata~\cite{MR21} showed that when $t = 2$, the upper bound holds for all $n \geq 2$, and moreover, if $n \geq 4$ and $f$ is the characteristic function of a $2$-setwise-intersecting family of size $2(n-2)!$, then $f$ has degree at most~$2$ and additionally satisfies $f^{=(n-2,1,1)} = 0$ (see \Cref{sec:degree-sym} for an explanation of this notation). In other words, $f^{=\lambda} \neq 0$ only for $\lambda = (n), (n-1,1), (n-2,2)$.

Behajaina, Maleki, Rasoamanana and Razafimahatratra~\cite{BMRR} showed that when $t = 3$, the upper bound holds for all $n \geq 11$, and moreover, if $f$ is the characteristic function of a $3$-setwise-intersecting family of size $6(n-3)!$, then $f^{=\lambda} \neq 0$ only for $\lambda = (n),(n-1,1),(n-2,2),(n-3,3)$.

In both cases, the authors were unable to classify the extremal families. Can we extend our techniques to show that the extremal families are $t$-setwise-cosets? The first step in this direction was already taken by the second author in her master's thesis~\cite{DafniThesis}, who showed that this holds for large $n$. One of the difficulties in obtaining bounds for small $n$ is taking advantage of the characterization of extremal $f$, which is stronger than a mere degree bound.

\smallskip

We mention in passing two more challenges. One is to extend the theory of setwise-intersecting families from permutations to perfect matchings. A $t$-setwise intersecting family of perfect matchings is one in which for any two perfect matchings $m_1,m_2$ there are two sets $A,B$ of size $t$ such that both $m_1,m_2$ match the vertices in $A$ to vertices in $B$. Another possible generalization is to require the existence of a single set $C$ of size $2t$ such that both $m_1$ and $m_2$ match vertices in $C$ to vertices in $C$.

The other is a common generalization of $t$-intersecting families and $t$-setwise-intersecting families. Given a partition $\lambda = \lambda_1,\ldots,\lambda_m$, a $\lambda$-intersecting family of permutations is one in which for any two permutations $\alpha,\beta$ there are disjoint sets $A_1,\ldots,A_m$ of sizes $\lambda_1,\ldots,\lambda_m$ such that $\alpha(A_1) = \beta(A_1),\ldots,\alpha(A_m)=\beta(A_m)$. A $t$-intersecting family corresponds to $\lambda = (1^t)$, and a $t$-setwise-intersecting family corresponds to $\lambda = (t)$.

\paragraph{Structure of the paper} We start with a few preliminaries in \Cref{sec:introduction}. Apart from introducing several key definitions and results, we also prove several results appearing in a more general form in~\cite{DFLLV}, in order to keep the paper self-contained. \Cref{thm:main-sym} is then proved in \Cref{sec:main-sym}, and \Cref{thm:main-pms} in \Cref{sec:main-pms}.

\paragraph{Acknowledgements} This project has received funding from the European Union's Horizon 2020 research and innovation programme under grant agreement No~802020-ERC-HARMONIC.

\section{Preliminaries} \label{sec:prel}

\subsection{Symmetric group} \label{sec:prel-sym}

For integer $n \geq 1$, the \emph{symmetric group} $S_n$ is the collection of all permutations of $[n] = \{1,\ldots,n\}$. We denote the identity permutation by $\id$.

A \emph{$t$-coset} of $S_n$ is a set of the form
\[
 \{ \alpha \in S_n : \alpha(i_1) = j_1, \ldots, \alpha(i_t) = j_t \},
\]
where $i_1,\ldots,i_t \in [n]$ are distinct and $j_1,\ldots,j_t \in [n]$ are distinct. A $t$-coset contains $(n-t)!$ permutations.

We use the term \emph{coset} for a $1$-coset. For $i,j \in [n]$, the coset $\coset{i}{j}$ consists of all permutations sending $i$ to $j$:
\[
 \coset{i}{j} = \{ \alpha \in S_n : \alpha(i) = j \}.
\]

Two permutations $\alpha,\beta \in S_n$ are \emph{$t$-intersecting} if there is a $t$-coset which contains both of them. Equivalently, the two permutations $t$-intersect if there are distinct indices $i_1,\ldots,i_t \in [n]$ such that $\alpha(i_1) = \beta(i_1),\ldots,\alpha(i_t) = \beta(i_t)$.

A subset $\mathcal{F} \subseteq S_n$ is \emph{$t$-intersecting} if every pair of permutations in $\mathcal{F}$ are $t$-intersecting.

\subsubsection{Degree} \label{sec:degree-sym}

We can represent permutations $\alpha \in S_n$ using $n^2$ variables $x_{ij}$ whose semantics are: $x_{ij} = 1$ if $\alpha(i) = j$, and $x_{ij} = 0$ otherwise. The \emph{degree} of a function $f\colon S_n \to \mathbb{R}$, denoted $\deg f$, is the minimal $d$ such that $f$ can be represented as a degree~$d$ polynomial over the variables $x_{ij}$.
For example, the characteristic function of a $t$-coset has degree at most~$t$, since it can be represented by the polynomial $x_{i_1j_1} \cdots x_{i_tj_t}$.

An equivalent way to define degree is via the representation theory of the symmetric group. Let $\mathbb{R}[S_n]$ be the vector space of all real-valued functions over the symmetric group. Representation theory gives an orthogonal decomposition (with respect to the inner product $\langle f,g \rangle = \sum_{\alpha \in S_n} f(\alpha) g(\alpha)$)
\[
 \mathbb{R}[S_n] = \bigoplus_{\lambda \vdash n} V^\lambda,
\]
where $\lambda$ goes over all integer partitions of $n$, and $V^\lambda$ are certain subspaces known as \emph{isotypic components}, which we define explicitly in \Cref{sec:degree-reduction-sym}. Accordingly, every function $f\colon S_n \to \mathbb{R}$ has a unique decomposition
\[
 f = \sum_{\lambda \vdash n} f^{=\lambda},
\]
where $f^{=\lambda} \in V^\lambda$. 
Ellis, Friedgut and Pilpel~\cite[Theorem 7]{EFP} showed that $\deg f$ is the maximal $d$ such that $f^{=\lambda} \neq 0$ for some $\lambda$ satisfying $\lambda_1 = n - d$. In particular, every function has degree at most $n-1$.

We can now formally state the main result of Meagher and Razafimahatratra~\cite[Corollary 5.5]{MR21}.

\begin{theorem} \label{thm:MR21}
If $n \geq 5$ then any $2$-intersecting family $\mathcal{F} \subseteq S_n$ contains at most $(n-2)!$ permutations. Furthermore, if $f\colon S_n \to \{0,1\}$ is the characteristic vector of a $2$-intersecting family of size $(n-2)!$, then $\deg f \leq 2$.
\end{theorem}

We will also need a result of Ellis, Friedgut and Pilpel~\cite{EFP} classifying degree~$1$ functions.

\begin{theorem}[{\cite[Corollary 2]{EFP}}] \label{thm:degree1-classification-sym}
If $f\colon S_n \to \{0,1\}$ has degree at most~$1$ then either
\[
 f = \sum_{j \in J} x_{ij}
\]
for some $i \in [n]$ and $J \subseteq [n]$, or
\[
 f = \sum_{i \in I} x_{ij}
\]
for some $I \subseteq [n]$ and $j \in [n]$.
\end{theorem}

See also \cite[Theorem 6.1]{DFLLV} for an alternative (but very similar) proof.

\subsubsection{Certificate complexity} \label{sec:certificate-sym}

Fix $n$. A \emph{certificate} is a set $C = \{(i_1,j_1),\ldots,(i_m,j_m)\} \subseteq [n] \times [n]$. A permutation $\alpha \in S_n$ \emph{satisfies} the certificate $C$ if $\alpha(i_1) = j_1,\ldots,\alpha(i_m) = j_m$. The set of permutations satisfying a certificate $C$ is either empty or a $|C|$-coset.

We sometimes think of a permutation $\alpha \in S_n$ as the certificate $\{(1,\alpha(1)),\ldots,(n,\alpha(n))\}$, which we call the \emph{certificate representation} of $\alpha$.

A \emph{Boolean function} is a function $f\colon S_n \to \{0,1\}$. Given $\alpha \in S_n$ such that $f(\alpha) = b$, a \emph{certificate for $\alpha$} (with respect to $f$) is a certificate $C$ such that:
\begin{enumerate}[(i)]
\item $\alpha$ satisfies $C$.
\item If $\beta$ satisfies $C$ then $f(\beta) = b$.
\end{enumerate}

Intuitively, in order to certify that $f(\alpha) = b$, it suffices to verify that $\alpha$ satisfies $C$. 

The \emph{certificate complexity} of $\alpha$, denoted $C(f,\alpha)$, is the minimum size of a certificate for $\alpha$. A certificate for $\alpha$ of this size is known as a \emph{minimum certificate}.
Since every permutation is determined by its values on the points $1,\ldots,n-1$, the certificate complexity of $\alpha$ is always at most $n-1$.

The \emph{certificate complexity} of $f$, denoted $C(f)$, is $\max_{\alpha \in S_n} C(f,\alpha)$. The certificate complexity of $f$ is always at most $n-1$.

Dafni et al.~\cite[Theorem 3.1]{DFLLV} showed that certificate complexity is polynomially related to the degree: $C(f) = O((\deg f)^8)$ and $\deg f = O(C(f)^4)$.

In the special case of degree~$1$, \Cref{thm:degree1-classification-sym} implies that $C(f) \leq 1$.

\begin{lemma} \label{lem:degree1-cert-sym}
If $f\colon S_n \to \{0,1\}$ has degree at most~$1$ then $C(f) \leq 1$.
\end{lemma}
\begin{proof}
According to \Cref{thm:degree1-classification-sym}, either $f = \sum_{j \in J} x_{ij}$ for some $i \in [n]$ and $J \subseteq [n]$, or $f = \sum_{i \in I} x_{ij}$ for some $I \subseteq [n]$ and $j \in [n]$. In the former case, every $\alpha \in S_n$ has the certificate $\{(i,\alpha(i))\}$, and in the latter case, every $\alpha \in S_n$ has the certificate $\{(\alpha^{-1}(j),j)\}$.
\end{proof}

\medskip

The following lemma, which essentially follows from~\cite[Lemma 7.3]{DFLLV}, shows that if $\mathcal{F}$ is a $2$-intersecting family then minimum certificates of any two permutations in $\mathcal{F}$ must $2$-intersect, unless $C(f) = n-1$.

\begin{lemma} \label{lem:2-int-coset-sym}
Let $f\colon S_n \to \{0,1\}$ be the characteristic function of a $2$-intersecting family, and suppose that $C(f) \leq n-2$. If $f(\alpha) = f(\beta) = 1$ and $C_\alpha,C_\beta$ are minimum certificates for $\alpha,\beta$ (respectively), then $|C_\alpha \cap C_\beta| \geq 2$.
\end{lemma}
\begin{proof}
We will show that there exist $\alpha' \in S_n$ satisfying $C_\alpha$ and $\beta' \in S_n$ satisfying $C_\beta$ such that $\alpha' \cap \beta' = C_\alpha \cap C_\beta$, where we identify $\alpha',\beta'$ with their certificate representations. Since $f(\alpha') = f(\beta') = 1$ and $f$ is the characteristic function of a $2$-intersecting family, we have $|\alpha' \cap \beta'| \geq 2$, and so $|C_\alpha \cap C_\beta| \geq 2$.

We start by constructing $\alpha'$ satisfying $C_\alpha$ such that $\alpha' \cap C_\beta = C_\alpha \cap C_\beta$. Let $C_\alpha = \{(i_1,j_1),\ldots,(i_m,j_m)\}$, where $m \leq n-2$, let $i'_1,\ldots,i'_{n-m}$ be the $i$-indices not mentioned in $C_\alpha$, and let $j'_1,\ldots,j'_{n-m}$ be the $j$-indices not mentioned in $C_\alpha$.

For each $i'_s$, there is at most one $j'_t$ such that $(i'_s,j'_t) \in C_\beta$. We can therefore arrange the indices in such a way that if $(i'_s,j'_t) \in C_\beta$ then $s = t$. We can now define $\alpha'$ explicitly:
\[
 \alpha'(i_1) = j_1, \ldots, \alpha'(i_m) = j_m,
 \alpha'(i'_1) = j'_2, \ldots, \alpha'(i'_{n-m-1}) = j'_{n-m}, \alpha'(i'_{n-m}) = j'_1.
\]

The same argument (replacing $C_\alpha,C_\beta$ by $C_\beta,\alpha'$) shows that we can find $\beta'$ satisfying $C_\beta$ such that $\alpha' \cap \beta' = \alpha' \cap C_\beta = C_\alpha \cap C_\beta$, completing the proof.
\end{proof}

When $C(f) = n-1$, the conclusion of \Cref{lem:2-int-coset-sym} indeed fails. For example, consider the family $\mathcal{F} \subseteq S_5$ given by $\mathcal{F} = \{ \id, (1\;2\;3) \}$, and the minimum certificates $\{(1,1),(2,2),(3,3),(4,4)\}$ and $\{(1,2),(2,3),(3,1),(4,4)\}$.

\smallskip

Using the concept of \emph{sensitivity}, we can rule out the case $C(f) = n-1$ for functions of degree at most~$2$. The proof is based on the technique used to prove \cite[Lemma 3.6]{DFLLV}.

\begin{lemma} \label{lem:C-bound-sym}
If $n \geq 8$ and $f\colon S_n \to \{0,1\}$ has degree at most~$2$, then $C(f) \leq n-2$.
\end{lemma}
\begin{proof}
The proof is by contradiction. Suppose that $n \geq 8$, that $f\colon S_n \to \{0,1\}$ has degree at most~$2$, and that $C(f) = n-1$. Without loss of generality, $C(f,\id) = n-1$ and $f(\id) = 0$. Then $f((i\;j)) = 1$  for all $i \neq j \in [n]$ (here $(i\;j)$ is the transposition switching $i$ and $j$), since otherwise $\{(k,k) : k \neq i,j\}$ would be a certificate for $\id$ of size $n-2$.

We construct a function $g\colon \{0,1\}^4 \to \{0,1\}$ as follows:
\[
 g(y_1,y_2,y_3,y_4) = f\bigl((1\;2)^{y_1} (3\;4)^{y_2} (5\;6)^{y_3} (7\;8)^{y_4}\bigr).
\]
Here $(i\;j)^0 = \id$, $(i\;j)^1 = (i\;j)$, and the input to $f$ is the product of four powers of this kind.

Since $f$ has degree at most~$2$, it can be written as a polynomial of degree at most~$2$ in the variables $x_{ij}$. Given $y_1,y_2,y_3,y_4$, the value of the variables $x_{ij}$ is:
\begin{align*}
    &x_{11} = x_{22} = 1-y_1 && x_{12} = x_{21} = y_1 \\
    &x_{33} = x_{44} = 1-y_2 && x_{34} = x_{43} = y_2 \\
    &x_{55} = x_{66} = 1-y_3 && x_{56} = x_{65} = y_3 \\
    &x_{77} = x_{88} = 1-y_4 && x_{78} = x_{87} = y_4 
\end{align*}
All other variables are assigned zero. This shows that $g$ can be expressed as a polynomial of degree at most~$2$ in the variables $y_1,y_2,y_3,y_4$. We say that $g$ has \emph{degree at most~$2$}. 

By construction, the function $g$ satisfies the following constraints:
\begin{align*}
    &g(0,0,0,0) = f(\id) = 0 \\
    &g(1,0,0,0) = f((1\;2)) = 1 \\
    &g(0,1,0,0) = f((3\;4)) = 1 \\
    &g(0,0,1,0) = f((5\;6)) = 1 \\
    &g(0,0,0,1) = f((7\;8)) = 1 
\end{align*}
We say that $g$ has \emph{sensitivity $4$}.

There are only $2^{16}$ many functions from $\{0,1\}^4$ to $\{0,1\}$. For each function, we can check whether it has degree at most~$2$ by solving linear equations, since each function has a unique representation as a multilinear polynomial in $y_1,y_2,y_3,y_4$ (this is a basic fact in Boolean function analysis). Going over all such functions in SAGE~\cite{sagemath}, we find out that none of them has sensitivity~$4$, a contradiction. The relevant code appears in \Cref{apx:C-bound-sym}. 
\end{proof}

For experts in Boolean function analysis, here is an alternative proof that $g$ cannot have sensitivity~$4$. Since $\deg g \leq 2$, every influential variable has influence at least $1/2$~\cite[Proposition 3.6]{ODonnell}. If $g$ has sensitivity~$4$, then it depends on $4$ variables, and so its total influence is~$2$. Since $g$ has degree~$2$, this can only happen if $g$ is homogeneous of degree~$2$. But in that case, the sensitivity of $g$ at every point is exactly~$2$ \cite[Proposition 3.7]{FHKL}.

The bound on the sensitivity cannot be improved: the Boolean function
\[
 g(y_1,y_2,y_3) = y_1 y_2 y_3 + (1 - y_1) (1 - y_2) (1 - y_3)
\]
has degree~$2$ and satisfies
\[
 g(0,0,0) = 1, \quad
 g(1,0,0) = g(0,1,0) = g(0,0,1) = 0.
\]

\subsubsection{Degree reduction} \label{sec:degree-reduction-sym}

Let $\mathcal{F} \subseteq S_n$. We denote the restriction of $\mathcal{F}$ to the coset $\coset{i}{j}$ by $\mathcal{F}|_{\coset{i}{j}}$. We can think of $\mathcal{F}|_{\coset{i}{j}}$ as a subset of $S_{n-1}$. Formally speaking, we can think of $S_n$ as the set of all injections from $[n]$ to $[n]$. The restriction of $S_n$ to the coset $\coset{i}{j}$ is isomorphic to the set of all injections from $[n] \setminus \{i\}$ to $[n] \setminus \{j\}$, which is the same as $S_{n-1}$ up to renumbering.

Let $f$ be the characteristic function of $\mathcal{F}$. The definition of degree using polynomials shows that $\deg f|_{\coset{i}{j}} \leq \deg f$, where $f$ is the characteristic function of $\mathcal{F}|_{\coset{i}{j}}$. Indeed, given a polynomial representing $f$, we can obtain a polynomial representing $f|_{\coset{i}{j}}$ by substituting $x_{ij} = 1$ and $x_{ij'} = x_{i'j} = 0$ for any $i' \neq i$ and $j' \neq j$.

The following crucial lemma, which forms part of the proof of~\cite[Lemma 5.7]{DFLLV}, states that if $\mathcal{F}$ is contained in $\coset{i}{j}$, then the degree of $f$ strictly decreases when restricting to $\coset{i}{j}$.

\begin{lemma} \label{lem:degree-reduction-sym}
Suppose that $f\colon S_n \to \{0,1\}$ is the characteristic function of a family which is a subset of the coset $\coset{i}{j}$. Then $\deg f|_{\coset{i}{j}} \leq \max( \deg f - 1, 0 )$.
\end{lemma}

The corresponding result for functions on the Boolean cube $\{0,1\}^n$ is easy to prove. Suppose that $f\colon \{0,1\}^n \to \{0,1\}$ is such that $f(x) = 1$ implies $x_n = 1$. We can write
\[
 f(x_1,\ldots,x_n) = (1-x_n) f_0(x_1,\ldots,x_{n-1}) + x_n f_1(x_1,\ldots,x_{n-1}).
\]
Substituting $x_n = 0$, we get $f_0 = 0$, and so $f = x_n f_1$. The restriction of $f$ to the coset $\{x : x_n = 1\}$ is $f_1$. If $f_1 \neq 0$ then clearly $\deg f = \deg f_1 + 1$, where $\deg f$ is the degree of the unique multilinear polynomial representing $f$.

The argument proving \Cref{lem:degree-reduction-sym} is more subtle. The starting point is the following explicit description of the subspaces $V^\lambda$, which we mentioned in \Cref{sec:degree-sym}.

Let $\lambda = \lambda_1,\ldots,\lambda_m \vdash n$. A \emph{Young tableau} of \emph{shape} $\lambda$ is an arrangement of the numbers $1,\ldots,n$ in $m$ rows of lengths $\lambda_1,\ldots,\lambda_m$, left justified. Here are some examples:
\[
 \begin{array}{ccc}
 \ytableaushort{341,2} && \ytableaushort{14,23} \\\\
 \lambda = (3,1) && \lambda = (2,2)
 \end{array}
\]

Let $s,t$ be a pair of Young tableaux of the same shape $\lambda_1,\ldots,\lambda_m \vdash n$. The function $e_{s,t}\colon S_n \to \{0,1\}$ is defined as follows: $e_{s,t}(\alpha) = 1$ if for every $k \in [m]$, $\alpha$ sends every number on the $k$-th row of $s$ to a number on the $k$-th row of $t$, and $e_{s,t}(\alpha) = 0$ otherwise. For example, consider
\[
 s = \ytableaushort{123,45} \quad
 t = \ytableaushort{315,24}
\]
Then $e_{s,t}(\alpha) = 1$ if $\alpha(1),\alpha(2),\alpha(3) \in \{3,1,5\}$ and $\alpha(4),\alpha(5) \in \{2,4\}$.

For a tableau $t$ with $m'$ columns, we denote by $C(t)$ the \emph{column stabilizer} of $t$, which is the set of all permutations of the entries of $t$ which only move entries within the same column. Each permutation $\pi \in C(t)$ thus consists of $m'$ permutations $\pi_1,\ldots,\pi_{m'}$, one for each column. The \emph{sign} of $\pi$, denote $(-1)^\pi$, is the product of the signs of the permutations $\pi_1,\ldots,\pi_{m'}$. We denote the result of applying $\pi$ to $t$ by $t^\pi$.

Let $s,t$ be a pair of Young tableaux of the same shape $\lambda \vdash n$. We define a function $\chi_{s,t}\colon S_n \to \{-1,0,1\}$ as follows:
\[
 \chi_{s,t} = \sum_{\pi \in C(t)} (-1)^\pi e_{s,t^\pi}.
\]
Continuing our above example, here is $t^\pi$ for all elements $\pi \in C(t)$, together with their sign:
\[
 \left(\ytableaushort{315,24}~,+1\right) \quad
 \left(\ytableaushort{345,21}~,-1\right) \quad
 \left(\ytableaushort{215,34}~,-1\right) \quad
 \left(\ytableaushort{245,31}~,+1\right)
\]

We are interested in the functions $\chi_{s,t}$ since they span $V^\lambda$.

\begin{theorem}[{\cite[Chapter 2]{Sagan}}] \label{thm:Vlambda-sym}
Let $\lambda \vdash n$. The subspace $V^\lambda$ is spanned by the functions $\chi_{s,t}$, where $s,t$ go over all pairs of Young tableaux of shape $\lambda$.
\end{theorem}

We can now prove \Cref{lem:degree-reduction-sym}.

\begin{proof}[Proof of \Cref{lem:degree-reduction-sym}]
Assume, for concreteness, that $i = j = n$.
If $n = 1$ then $\deg f|_{\coset{n}{n}} = 0$, so there is nothing to prove. If $f|_{\coset{n}{n}} = 0$ then again $\deg f|_{\coset{n}{n}} = 0$, so there is nothing to prove. Therefore we can assume that $n \geq 2$ and that $f|_{\coset{n}{n}} \neq 0$.

Let $d = \deg f|_{\coset{n}{n}}$. According to the spectral definition of degree (see \Cref{sec:degree-sym}), $f|_{\coset{n}{n}}^{=\lambda} \neq 0$ for some $\lambda \vdash n-1$ such that $\lambda_1 = (n-1) - d$. Since the decomposition into isotypic components is orthogonal, this implies that $\langle f|_{\coset{n}{n}}, f|_{\coset{n}{n}}^{=\lambda} \rangle = \langle f|_{\coset{n}{n}}^{=\lambda}, f|_{\coset{n}{n}}^{=\lambda} \rangle \neq 0$. Since $f|_{\coset{n}{n}}^{=\lambda} \in V^\lambda$, \Cref{thm:Vlambda-sym} implies that $\langle f|_{\coset{n}{n}}, \chi_{s,t} \rangle \neq 0$ for some Young tableaux $s,t$ of shape $\lambda$.

Let $s',t'$ be the Young tableaux obtained by adding a new row to $s,t$ (respectively) consisting only of the number $n$. Here is an example:
\[
 s = \ytableaushort{123,45} \quad
 t = \ytableaushort{315,24} \quad \longrightarrow \quad
 s' = \ytableaushort{123,45,6} \quad
 t' = \ytableaushort{315,24,6}
\]
The tableaux $s',t'$ have shape $\mu$ satisfying $\mu_1 = \lambda_1 = (n-1)-d = n-(d+1)$.
We will show that $\langle f, \chi_{s',t'} \rangle \neq 0$. The orthogonality of the decomposition into isotypic components would imply that $f^{=\mu} \neq 0$, and so $\deg f \geq d+1$ by the spectral definition of degree.

Let $\pi' \in C(t')$ be any permutation such that $\pi'(n) \neq n$. By assumption, $f(\alpha') = 1$ only if $\alpha'(n) = n$, in which case $e_{s',(t')^{\pi'}}(\alpha') = 0$. Therefore $\langle f, e_{s',(t')^{\pi'}} \rangle = 0$. This shows that
\[
 \langle f, \chi_{s',t'} \rangle = \sum_{\substack{\pi' \in C(t') \\ \pi'(n) = n}} (-1)^{\pi'} \langle f, e_{s',(t')^{\pi'}} \rangle.
\]
If $\pi'(n) = n$ then $e_{s',(t')^{\pi'}}(\alpha') \neq 0$ only if $\alpha'(n) = n$. Conversely, if $\alpha'(n) = n$, that is, if $\alpha' \in \coset{n}{n}$, then $e_{s',(t')^{\pi'}}(\alpha') = e_{s,t^\pi}(\alpha)$, where $\alpha,\pi$ are the restrictions of $\alpha',\pi'$ to $[n-1]$. Moreover, $(-1)^\pi = (-1)^{\pi'}$. Therefore
\[
 \langle f, \chi_{s',t'} \rangle = \sum_{\pi \in C(t)} (-1)^\pi \langle f|_{\coset{n}{n}}, e_{s,t^\pi} \rangle = \langle f|_{\coset{n}{n}}, \chi_{s,t} \rangle \neq 0. \qedhere
\]
\end{proof}

\subsection{Perfect matching scheme} \label{sec:prel-pms}

For integer $n \ge 1$, the \emph{perfect matching scheme} $\cM_{2n}$ is the collection of all perfect matchings of the complete graph $K_{2n}$. Since the symmetric group $S_n$ can be viewed as the collection of all perfect matchings of the complete \emph{bipartite} graph $K_{n,n}$, we can think of the perfect matching scheme as the non-bipartite analog of the symmetric group.

The perfect matching scheme $\cM_{2n}$ contains $(2n-1)!!$ many different perfect matchings. Here
\[
 (2n-1)!! = (2n-1)(2n-3)(2n-5)\cdots(5)(3)(1) = \frac{(2n)!}{2^nn!}.
\]

We sometimes think of a perfect matching in $\cM_{2n}$ as a set of $n$ pairs of elements from $[2n]$, which together cover all of $[2n]$ (the \emph{pair representation}). At other times, it will be useful to think of a perfect matching in $\cM_{2n}$ as a fixed-point-free involution on $[2n]$, that is, $m(i) \neq i$ is the element that $i$ is matched to, and $m(m(i)) = i$.

A \emph{$t$-coset} of $\cM_{2n}$ is a set of the form
\[
 \{m \in \cM_{2n} : m(i_1) = j_1, \ldots, m(i_t) = j_t\},
\]
where $i_1,\ldots,i_t,j_1,\ldots,j_t \in [2n]$ are distinct.\footnote{This is a slight abuse of terminology as $\cM_{2n}$ doesn't afford a natural group structure.} 
A $t$-coset contains $(2(n-t)-1)!!$ perfect matchings.

We use the term \emph{coset} for a $1$-coset. For $i \neq j \in [2n]$, the coset $\dcoset{i}{j}$ consists of all perfect matchings in which $i,j$ are matched:
\[
 \dcoset{i}{j} = \{m \in \cM_{2n} : m(i) = j\}.
\]
Note that $\dcoset{i}{j} = \dcoset{j}{i}$.

Two perfect matchings $m_1,m_2 \in \cM_{2n}$ are \emph{$t$-intersecting} if there is a $t$-coset which contains both of them. A subset $\mathcal{F} \subseteq \cM_{2n}$ is \emph{$t$-intersecting} if any two perfect matchings in $\mathcal{F}$ are $t$-intersecting.

\subsubsection{Degree} \label{sec:degree-pms}

We can represent perfect matchings $m \in \cM_{2n}$ using $\binom{2n}{2}$ variables $x_{ij}$ (where the index is an unordered pair of elements) whose semantics are: $x_{ij} = 1$ if $m(i) = j$, and $x_{ij} = 0$ otherwise. The \emph{degree} of a function $f\colon \cM_{2n} \to \mathbb{R}$, denoted $\deg f$, is the minimal $d$ such that $f$ can be represented as a degree~$d$ polynomial over the variables $x_{ij}$. For exampole, the characteristic function of a $t$-coset has degree at most $t$.

An equivalent way to define degree is via the representation theory of the perfect matching scheme. Let $\mathbb{R}[\cM_{2n}]$ be the vector space of all real-valued functions over the perfect matching scheme. Representation theory (for example, \cite[Chapter 11]{CSST08}) gives an orthogonal decomposition (with respect to the inner product $\langle f, g \rangle = \sum_{m \in \cM_{2n}} f(m) g(m)$)
\[
 \mathbb{R}[\cM_{2n}] = \bigoplus_{\lambda \vdash n} V^{2\lambda},
\]
where $\lambda$ goes over all integer partitions of $n$, $2\lambda$ is the partition of $2n$ obtained by doubling each part, and $V^{2\lambda}$ are the isotypic components, which we define explicitly in \Cref{sec:degree-reduction-pms}. Accordingly, every function $f\colon \cM_{2n} \to \mathbb{R}$ has a unique decomposition
\[
 f = \sum_{\lambda \vdash n} f^{=\lambda},
\]
where $f^{=\lambda} \in V^{2\lambda}$. Lindzey~\cite[Theorem 5.1.1]{LindzeyThesis} showed that $\deg f$ is the maximal $d$ such that $f^{=\lambda} \neq 0$ for some $\lambda$ satisfying $\lambda_1 = n - d$. In particular, every function has degree at most $n - 1$.

We can now formally state the main result of Fallat, Meagher and Shirazi~\cite[Theorem 4.13]{FMS}.

\begin{theorem} \label{thm:FMS}
If $n \geq 3$ then any $2$-intersecting family $\mathcal{F} \subseteq \cM_{2n}$ contains at most $(2n-5)!!$ perfect matchings. Furthermore, if $f\colon \cM_{2n} \to \{0,1\}$ is the characteristic vector of a $2$-intersecting family of size $(2n-5)!!$, then $\deg f \leq 2$.
\end{theorem}

We will also need a result of Dafni et al.~\cite{DFLLV} classifying degree~$1$ functions.

\begin{theorem}[{\cite[Theorem 6.2]{DFLLV}}] \label{thm:degree1-classification-pms}
If $f\colon \cM_{2n} \to \{0,1\}$ has degree at most $1$ then $f$ is either a \emph{dictator}:
\[
 f = \sum_{j \in J} x_{ij},
\]
for some $i \in [2n]$ and $J \subseteq [2n]$ not containing $i$; or $f$ is a \emph{triangle}:
\[
 f = x_{ij} + x_{ik} + x_{jk},
\]
for some distinct $i,j,k \in [2n]$; or $f$ is an \emph{anti-triangle}:
\[
 f = 1 - x_{ij} - x_{ik} - x_{jk},
\]
for some distinct $i,j,k \in [2n]$.
\end{theorem}

\subsubsection{Certificate complexity} \label{sec:certificate-pms}

Fix $n$. A \emph{certificate} is a set $C = \bigl\{ \{i_1,j_1\}, \ldots, \{i_r,j_r\} \bigr\} \subseteq \binom{[2n]}{2}$, where $\binom{[2n]}{2}$ is the collection of all subsets of $[2n]$ of size~$2$. A perfect matching $m \in \cM_{2n}$ \emph{satisfies} the certificate $C$ if $m(i_1) = j_1,\ldots,m(i_r) = j_r$.

A \emph{Boolean function} is a function $f\colon \cM_{2n} \to \{0,1\}$. Just as in \Cref{sec:certificate-sym}, for each $m \in \cM_{2n}$ we can define the following notions: a certificate for $m$ (with respect to $f$); the certificate complexity of $m$, denoted $C(f,m)$; a minimum certificate; and the certificate complexity of $f$, denoted $C(f)$. As in the case of the symmetric group, $C(f) \leq n-1$, since a perfect matching is determined by any $n-1$ of its edges. Dafni et al.~\cite{DFLLV} showed that $C(f)$ is polynomially related to $\deg f$.

In the case of the symmetric group, degree~$1$ functions have certificate complexity~$1$. This is no longer the case for the perfect matching scheme, since triangles and anti-triangles have certificate complexity~$2$. 

\begin{lemma} \label{lem:degree1-cert-pms}
If $f\colon \cM_{2n} \to \{0,1\}$ has degree at most~$1$ then $C(f) \leq 2$.
\end{lemma}
\begin{proof}
According to \Cref{thm:degree1-classification-pms}, $f$ is either a dictator, a triangle, or an anti-triangle. If $f = \sum_{j \in J} x_{ij}$ is a dictator, then every $m$ has a certificate $\bigl\{ \{i,m(i)\} \bigr\}$. If $f = x_{ij} + x_{ik} + x_{jk}$ is a triangle, then every $m$ such that $f(m) = 1$ has one of the certificates $\bigl\{\{i,j\}\bigr\},\bigl\{\{i,k\}\bigr\},\bigl\{\{j,k\}\bigr\}$, and every $m$ such that $f(m) = 0$ has the certificate $\bigl\{\{i,m(i)\},\{j,m(j)\}\bigr\}$. If $f$ is an anti-triangle then $1-f$ is a triangle, so we can use the same certificates as in the case of a triangle.
\end{proof}

The proof of \Cref{lem:2-int-coset-sym} extends to the perfect matching scheme with minor changes.

\begin{lemma} \label{lem:2-int-coset-pms}
Let $f\colon \cM_{2n} \to \{0,1\}$ be the characteristic function of a $2$-intersecting family, and suppose that $C(f) \leq n-2$. If $f(m_1) = f(m_2) = 1$ and $C_{m_1},C_{m_2}$ are minimum certificates for $m_1,m_2$ (respectively), then $|C_{m_1} \cap C_{m_2}| \geq 2$.
\end{lemma}
\begin{proof}
We will show how to construct a perfect matching $m'_1$ satisfying $C_{m_1}$ such that $m'_1 \cap C_{m_2} = C_{m_1} \cap C_{m_2}$. The same argument can be used to construct a perfect matching $m'_2$ satisfying $C_{m_2}$ such that $m'_1 \cap m'_2 = m'_1 \cap C_{m_2}$. It follows that $|C_{m_1} \cap C_{m_2}| = |m'_1 \cap m'_2| \geq 2$, since $f$ is the characteristic function of a $2$-intersecting family.

Let $C_{m_1} = \bigl\{\{i_1,j_1\},\ldots,\{i_r,j_r\}\bigr\}$, where $r \leq n-2$, and let $k_1,\ldots,k_{2(n-r)}$ be the indices not mentioned in $C_{m_1}$. Rearrange the indices so that if $\{k_s,k_t\} \in C_{m_2}$ then $\{s,t\} = \{2\ell-1,2\ell\}$ for some $\ell \in [n-r]$. The perfect matching $m'_1$ has the pair representation
\[
 \{i_1,j_1\},\ldots,\{i_r,j_r\},\{k_1,k_{n-r+1}\},\{k_2,k_{n-r+2}\},\ldots,\{k_{n-r},k_{2(n-r)}\}.
\]
Since $n-r \geq 2$, the new pairs do not intersect $C_{m_2}$, completing the proof.
\end{proof}

Using sensitivity, we can rule out the case $C(f) = n-1$ for functions of degree at most~$2$, just as we did for the symmetric group.

\begin{lemma} \label{lem:C-bound-pms}
If $n \geq 8$ and $f\colon \cM_{2n} \to \{0,1\}$ has degree at most~$2$, then $C(f) \leq n-2$.
\end{lemma}
\begin{proof}
The proof is by contradiction. Suppose that $n \geq 8$, that $f\colon \cM_{2n} \to \{0,1\}$ has degree at most~$2$, and that $C(f) = n-1$, say $C(f,m) = n-1$ for $m = \bigl\{\{1,2\},\{3,4\},\ldots,\{2n-1,2n\}\bigr\}$. Since $C(f,m) > n-2$, $m \setminus \bigl\{\{2i-1,2i\},\{2j-1,2j\}\bigr\}$ is not a certificate for $m$ for any $i \neq j \in [n]$, and so $f(m^{ij}) \neq f(m)$, where $m^{ij}$ is obtained from $m$ by replacing $\bigl\{\{2i-1,2i\},\{2j-1,2j\}\bigr\}$ with $\bigl\{\{2i-1,2j\},\{2j-1,2i\}\bigr\}$.

We construct a function $g\colon \{0,1\}^4 \to \{0,1\}$ as follows: $g(y_1,y_2,y_3,y_4)$ is the value of $f$ on the perfect matching obtained from $m$ by applying the subset of the operations $^{12},^{34},^{56},^{78}$ specified by the input (for example, if $y_1 = 1$ then we apply $^{12}$). As in the proof of \Cref{lem:C-bound-sym}, this results in a function of degree at most~$2$ satisfying $g(0,0,0,0) \neq g(1,0,0,0) = g(0,1,0,0) = g(0,0,1,0) = g(0,0,0,1)$, which is impossible.
\end{proof}

\subsubsection{Degree reduction} \label{sec:degree-reduction-pms}

Let $\mathcal{F} \subseteq \cM_{2n}$. We denote the restriction of $\mathcal{F}$ to the coset $\dcoset{i}{j}$ by $\mathcal{F}|_{\dcoset{i}{j}}$, which we can think of as a subset of $\cM_{2(n-1)}$.

Let $f$ be the characteristic function of $\mathcal{F}$. The definition of degree using polynomials shows that $\deg f|_{\dcoset{i}{j}} \leq \deg f$. The following analog of \Cref{lem:degree-reduction-sym} shows that if $\mathcal{F}$ is contained in $\dcoset{i}{j}$, then the degree of $f$ strictly decreases when restriction to $\dcoset{i}{j}$.

\begin{lemma} \label{lem:degree-reduction-pms}
Suppose that $f\colon \mathcal{M}_{2n} \to \{0,1\}$ is the characteristic function of a family which is a subset of the coset $\dcoset{i}{j}$. Then $\deg f|_{\dcoset{i}{j}} \leq \max( \deg f - 1, 0 )$.
\end{lemma}

Before we can prove this lemma, we need an explicit description of the subspaces $V^{2\lambda}$, which were mentioned in \Cref{sec:degree-pms}.

Let $\lambda \vdash n$, and let $t$ be a Young tableau of shape $2\lambda$. The function $e_t\colon \cM_{2n} \to \{0,1\}$ is defined as follows: $e_t(m) = 1$ if for all $i \in [2n]$, the element $m(i)$ is in the same row of $t$ as $i$, and $e_t(m) = 0$ otherwise. For example, consider
\[
 t = \ytableaushort{1234,5678}
\]
Then $e_t(m) = 1$ if $m(1),m(2),m(3),m(4) \in \{1,2,3,4\}$ and $m(5),m(6),m(7),m(8) \in \{5,6,7,8\}$.

Given a Young tableau $t$ of shape $2\lambda$, we define a function $\chi_t\colon \mathcal{M}_{2n} \to \{-1,0,1\}$ as follows:
\[
 \chi_t = \sum_{\pi \in C(t)} (-1)^\pi e_{t^\pi}.
\]
These functions span $V^{2\lambda}$.

\begin{theorem}[{\cite[Theorem 5.2.6]{LindzeyThesis}}] \label{thm:Vlambda-pms}
Let $\lambda \vdash n$. The subspace $V^{2\lambda}$ is spanned by the functions $\chi_t$, where $t$ goes over all Young tableaux of shape $2\lambda$.
\end{theorem}

We can now prove \Cref{lem:degree-reduction-pms}. The proof is along the lines of \Cref{lem:degree-reduction-sym}.

\begin{proof}[Proof of \Cref{lem:degree-reduction-pms}]
Assume, for concreteness, that $i=2n-1$ and $j=2n$. If $n = 1$ then $\deg f|_{\dcoset{2n-1}{2n}} = 0$, so there is nothing to prove. Similarly, if $f|_{\dcoset{2n-1}{2n}} = 0$ then there is nothing to prove. Therefore we can assume that $n \geq 2$ and that $f|_{\dcoset{2n-1}{2n}} \neq 0$.

Let $d = deg f|_{\dcoset{2n-1}{2n}}$. According to the spectral definition of degree (see \Cref{sec:degree-pms}), $f|_{\dcoset{2n-1}{2n}}^{=\lambda} \neq 0$ for some $\lambda \vdash n-1$ such that $\lambda_1 = (n-1)-d$. Since the decomposition into isotypic components is orthogonal, \Cref{thm:Vlambda-pms} implies that $\langle f|_{\dcoset{2n-1}{2n}}, \chi_t \rangle \neq 0$ for some Young tableau $t$ of shape $2\lambda$. Among all tableau $t$ with $2(n-1-d)$ squares on the first lines satisfying $\langle f|_{\dcoset{2n-1}{2n}}, \chi_t \rangle \neq 0$, we choose one which maximizes the number of rows (which is the number of parts in the corresponding partition), say $t$ has $R$ rows.

Let $t'$ be the Young tableau obtained by adding a new row to $t$ consisting only of the numbers $2n-1,2n$. The tableau $t'$ has shape $\mu$ satisfying $\mu_1 = \lambda_1 = n - (d+1)$. We will show that $\langle f, \chi_{t'} \rangle \neq 0$, and so $\deg f \geq d+1$ by the spectral definition of degree.

If $\pi' \in C(t')$ is any permutation such that $2n-1,2n$ do not end up in the same row, then $\langle f, e_{(t')^{\pi'}} \rangle = 0$, since $e_{(t')^{\pi'}}(m) \neq 0$ implies that $m(2n-1) \neq 2n$, and so $f(m) = 0$. Therefore
\[
 \langle f, \chi_{t'} \rangle = \sum_{r=1}^{R+1} \sum_{\substack{\pi' \in C(t') \\ 2n-1,2n \text{ are on} \\ \text{row $r$ of $(t')^{\pi'}$}}} (-1)^{\pi'} \langle f, e_{(t')^{\pi'}} \rangle.
\]
For $r \leq R$, let $t^r$ be the tableau obtained from $t$ by taking the first two elements of row $r$, moving them to a new row at then very end, removing row $r$ if it is empty, and otherwise sliding the remainder of row $r$ to the left. For example:
\[
 t = \ytableaushort{1234,56,78} \quad
 t^1 = \ytableaushort{34,56,78,12} \quad
 t^2 = \ytableaushort{1234,78,56} \quad
 t^3 = t
\]
We also define $t^{R+1} = t$.
Observe that
\[
 \langle f, \chi_{t'} \rangle = \sum_{r=1}^{R+1} \sum_{\pi \in C(t^r)} (-1)^\pi \langle f|_{\dcoset{2n-1}{2n}}, e_{(t^r)^\pi} \rangle = \sum_r \langle f|_{\dcoset{2n-1}{2n}}, \chi_{t^r} \rangle.
\]
If the $r$-th row of $t$ contains more than two elements then $t^r$ contains more rows than $t$, and so by the choice of $t$, we have $\langle f|_{\dcoset{2n-1}{2n}}, \chi_{t^r} \rangle = 0$. Otherwise, $\chi_{t^r}$ differs from $\chi_t$ by an identical permutation of the first two columns, and so $\langle f|_{\dcoset{2n-1}{2n}}, \chi_{t^r} \rangle = \langle f|_{\dcoset{2n-1}{2n}}, \chi_t \rangle$. Therefore, denoting by $\ell$ the number of rows of $t$ of length~$2$, we have
\[
 \langle f, \chi_{t'} \rangle = (\ell+1) \langle f|_{\dcoset{2n-1}{2n}}, \chi_t \rangle \neq 0.
\]
The reason we get $\ell+1$ rather than $\ell$ is the choice $r = R+1$, which always leads to $t^{R+1} = t$.
\end{proof}

\section{Main theorem: symmetric group} \label{sec:main-sym}

\subsection{Structure theorem} \label{sec:structure-sym}

We start with a simple observation on $2$-intersecting families, whose statement requires a few definitions.

Two cosets $\coset{i_1}{j_1},\coset{i_2}{j_2}$ are \emph{compatible} if their intersection is non-empty. In other words, either the two cosets are equal, or $i_1 \neq i_2$ and $j_1 \neq j_2$.
A family $\mathcal{F} \subseteq S_n$ is \emph{$m$-covered}, for some integer $m \geq 1$, if there are $m$ compatible cosets $\coset{i_1}{j_1},\ldots,\coset{i_m}{j_m}$ such that $\mathcal{F} \subseteq \coset{i_1}{j_1} \cup \cdots \cup \coset{i_m}{j_m}$. 

\begin{lemma} \label{lem:covering-sym}
Let $f$ be the characteristic function of a $2$-intersecting family $\mathcal{F} \subseteq S_n$, and suppose that $C(f) \leq n-2$.

If $\alpha \in \mathcal{F}$ then $\mathcal{F}$ is $m$-covered for $m = C(f,\alpha) - 1$.
\end{lemma}
\begin{proof}
Let $C_\alpha = \{(i_1,j_1),\ldots,(i_{m+1},j_{m+1})\}$ be a minimal certificate for $\alpha$. We will show that $\mathcal{F} \subseteq \coset{i_1}{j_1} \cup \cdots \cup \coset{i_m}{j_m}$.

Suppose that $\beta \in \mathcal{F}$, and let $C_\beta$ be a minimal certificate for $\beta$. According to \Cref{lem:2-int-coset-sym}, $|C_\alpha \cap C_\beta| \geq 2$. Therefore $(i_s,j_s) \in C_\beta$ for some $s \in [m]$, implying that $\beta \in \coset{i_s}{j_s} \subseteq \coset{i_1}{j_1} \cup \cdots \cup \coset{i_m}{j_m}$.
\end{proof}

The following structure lemma provides a kind of converse to \Cref{lem:covering-sym}.

\begin{lemma} \label{lem:structure-sym}
Let $\mathcal{F} \subseteq S_n$ be an $m$-covered family, say $\mathcal{F} \subseteq \coset{i_1}{j_1} \cup \cdots \cup \coset{i_m}{j_m}$. Suppose that its characteristic function $f$ has degree at most~$2$.

Every $\alpha$ has a certificate of the form $\{(i_1,\alpha(i_1)),\ldots,(i_m,\alpha(i_m)),(i,j)\}$, and in particular, $C(f) \leq m + 1$.
\end{lemma}
\begin{proof}
We will prove by induction on $m$ that if $\mathcal{F} \subseteq S_n$ is $m$-covered, say $\mathcal{F} \subseteq \coset{i_1}{j_1} \cup \cdots \cup \coset{i_m}{j_m}$, and its characteristic function $f$ has degree at most~$2$, then for every distinct $k_1,\ldots,k_m \in [n]$, the restriction of $f$ to $\coset{i_1}{k_1} \cap \cdots \cap \coset{i_m}{k_m}$ has degree at most~$1$. By \Cref{lem:degree1-cert-sym}, this restriction has certificate complexity at most~$1$. For any $\alpha \in S_n$, choosing $k_1 = \alpha(i_1),\ldots,k_m = \alpha(i_m)$, we deduce that $\alpha$ has a certificate of the required form.

We start with the base case $m = 1$. Let $\mathcal{F} \subseteq S_n$ be such that $\mathcal{F} \subseteq \coset{i_1}{j_1}$, and suppose that its characteristic function $f$ has degree at most~$2$. \Cref{lem:degree-reduction-sym} directly implies that the restriction of $f$ to $\coset{i_1}{j_1}$ has degree at most~$1$, as required.

Suppose now that $m \geq 2$. Let $\mathcal{F} \subseteq S_n$ be such that $\mathcal{F} \subseteq \coset{i_1}{j_1} \cup \cdots \cup \coset{i_m}{j_m}$, and suppose that its characteristic function $f$ has degree at most~$2$. Consider any distinct $k_1,\ldots,k_m \in [n]$. We need to show that the restriction of $f$ to $\coset{i_1}{k_1} \cap \cdots \cap \coset{i_m}{k_m}$ has degree at most~$1$.

Suppose first that $k_s \neq j_s$ for some $s \in [m]$, say $k_m \neq j_m$. The restriction of $\mathcal{F}$ to $\coset{i_m}{k_m}$ is contained in $\coset{i_1}{j_1} \cup \cdots \cup \coset{i_{m-1}}{j_{m-1}}$, and its characteristic function has degree at most~$2$. Therefore the induction hypothesis implies that the restriction of $f$ to $\coset{i_1}{k_1} \cap \cdots \cap \coset{i_m}{k_m}$ has degree at most~$1$.

It remains to show that the restriction of $f$ to $\coset{i_1}{j_1} \cap \cdots \cap \coset{i_m}{j_m}$ has degree at most~$1$. To see this, consider the function
\[
 g := x_{i_1j_1} \cdots x_{i_mj_m} f = f - \sum_{\substack{k_1,\ldots,k_m \\ (k_1,\ldots,k_m) \neq (j_1,\ldots,j_m)}} x_{i_1k_1} \cdots x_{i_mk_m} f.
\]
We showed above that the restriction of $f$ to $\coset{i_1}{k_1} \cap \cdots \cap \coset{i_m}{k_m}$ has degree at most~$1$ whenever $(k_1,\ldots,k_m) \neq (j_1,\ldots,j_m)$. Therefore every term in the sum on the right-hand side has degree at most $m+1$, since it can be represented as $x_{i_1k_1} \cdots x_{i_mk_m}$ times the linear polynomial representing the restriction of $f$ to $\coset{i_1}{k_1} \cap \cdots \cap \coset{i_m}{k_m}$. Consequently, $\deg g \leq \max(\deg f, m+1) = m+1$.

By construction, $g$ is the characteristic function of a family which is contained in the coset intersection $\mathcal{I} = \coset{i_1}{j_1} \cap \cdots \cap \coset{i_m}{j_m}$. Applying \Cref{lem:degree-reduction-sym} repeatedly, we conclude that the restriction of $g$ to $\mathcal{I}$ has degree at most $\max(m+1-m,0) = 1$. Since the restriction of $f$ to $\mathcal{I}$ is the same as the restriction of $g$ to $\mathcal{I}$, this completes the proof of the induction step.
\end{proof}

\subsection{Proof of main theorem} \label{sec:proof-sym}

We are now ready to prove our main upper bound, from which \Cref{thm:main-sym} will quickly follow.

\begin{lemma} \label{lem:main-ub-sym}
Let $\mathcal{F} \subseteq S_n$ be a $2$-intersecting family whose characteristic function $f$ has degree at most~$2$. Suppose that $C(f) \leq n-2$.

The family $\mathcal{F}$ is contained in the union of at most $T$ many $C(f)$-cosets, where
\[
 T = \frac{2\lfloor C(f)/2 \rfloor (C(f) - 1)!}{2^{\lfloor C(f)/2 \rfloor}}.
\]
Moreover, $C(f) \geq 2$.
\end{lemma}
\begin{proof}
If $\mathcal{F}$ is empty then the lemma is vacuously true. Otherwise, choose an arbitrary $\beta \in \mathcal{F}$. \Cref{lem:2-int-coset-sym}, applied to $\alpha = \beta$, shows that $C(f,\beta) \geq 2$. \Cref{lem:covering-sym} shows that $\mathcal{F}$ is $m$-covered, where $m = C(f,\beta) - 1$. Without loss of generality, we can assume that $\mathcal{F} \subseteq \coset{1}{1} \cup \cdots \cup \coset{m}{m}$.

\Cref{lem:structure-sym} shows that $C(f) \leq m + 1$, and so $m = C(f) - 1$ and $C(f) = C(f,\beta) \geq 2$. The \lcnamecref{lem:structure-sym} also shows that every $\alpha \in \mathcal{F}$ has a certificate $C_\alpha$ of the form $\{(1,\alpha(1)),\ldots,(m,\alpha(m)),(i,j)\}$.
It also implies that $\mathcal{F}$ is not $(m-1)$-covered (otherwise its certificate complexity would be at most $m < C(f)$). Therefore, if $X$ is any set of at most $m-1$ compatible cosets, then there is $\alpha_X \in \mathcal{F}$ which is not contained in the union of the cosets in $X$. Fix a certificate $C_X$ for $\alpha_X$ of the form $\{(1,\alpha_X(1)),\ldots,(m,\alpha_X(m)),(i_X,j_X)\}$.

We think of the certificates $C_\alpha,C_X$ as colored using $m + 1$ colors: for $k \in [m]$, a pair $(k,\ast)$ is colored $k$, and the remaining pair $(i,j)$ is colored $m+1$. We can determine the color of a pair $(i',j')$ given only the pair: if $i' \in [m]$ then the color is $i'$, and otherwise the color is $m+1$.

We will prove the lemma by showing that if $\alpha \in \mathcal{F}$ then $C_\alpha$ is one of at most $T$ possible certificates. Since $C_\alpha$ corresponds to a $C(f)$-coset containing $\alpha$, the lemma would follow.
In the proof, we will freely identify a pair $(i,j)$ with the corresponding coset $\coset{i}{j}$, and from now on we will only ever mention cosets.

\smallskip

The proof is slightly different depending on the parity of $C(f)$. We start with the case in which $C(f) = 2r$ is even. Let $\alpha \in \mathcal{F}$.
We define a sequence $X_0 \subset \cdots \subset X_r \subseteq C_\alpha$, where $X_s$ is a collection of $2s$ compatible cosets, as follows. The starting point is $X_0 = \emptyset$. Now let $s < r$, and suppose that $X_s$ has been defined.  Since $|X_s| = 2s \leq C(f) - 2 = m - 1$, there is a permutation $\alpha_{X_s} \in \mathcal{F}$ which lies outside of the union of the cosets in $X_s$. According to \Cref{lem:2-int-coset-sym}, the certificates $C_\alpha$ and $C_{X_s}$ must have at least two cosets in common $q_1,q_2$. The choice of $\alpha_{X_s}$ guarantees that $q_1,q_2 \notin X_s$, and we define $X_{s+1} = X_s \cup \{q_1,q_2\}$. By construction, $X_{s+1} \subseteq C_\alpha$, and so the cosets in $X_{s+1}$ are indeed compatible.

Note that $|X_r| = 2r = |C_\alpha|$, and so $X_r = C_\alpha$. For $s < r$, given $X_s$, the set $X_{s+1} \setminus X_s$ consists of two cosets belonging to $C_{X_s}$. Moreover, since $X_{s+1} \subseteq C_\alpha$ and $C_\alpha$ contains one coset of each color, the colors of the two cosets in $X_{s+1} \setminus X_s$ are different from the colors of the cosets in $X_s$. Since $C_{X_s}$ also contains one coset of each color, this means that there are $\binom{|C_{X_s}| - |X_s|}{2} = \binom{C(f) - 2s}{2}$ choices for $X_{s+1} \setminus X_s$. In total, the number of possible choices for $C_\alpha$ is
\[
 \prod_{s=0}^{r-1} \binom{C(f) - 2s}{2} =
 \frac{C(f)!}{2^r},
\]
matching the formula for $T$ in the statement of the lemma.

\smallskip

Now suppose that $C(f) = 2r+1$ is odd. Let $\alpha \in \mathcal{F}$. We define a sequence $X_0 \subset \cdots \subset X_r \subseteq C_\alpha$, where $X_s$ is a collection of $2s+1$ compatible cosets, as follows. We start by defining $X_0$. According to \Cref{lem:2-int-coset-sym}, the certificates $C_\alpha$ and $C_\emptyset$ must contain at least two cosets in common $q_1,q_2$. At least one of these cosets is of the form $(k,\alpha(k))$ for some $k \in [m]$. We define $X_0 = \{(k,\alpha(k))\}$. Now let $s < r$, and suppose that $X_s$ has been defined. Since $|X_s| = 2s+1 \leq C(f) - 2 = m - 1$, we can define $X_{s+1}$ as in the case of even $C(f)$.

Note that $|X_r| = 2r+1 = |C_\alpha|$, and so $X_r = C_\alpha$. The certificate $X_0$ contains a single coset $(k,\ell) \in C_\emptyset$, where $k \in [m]$, and so there are $m = C(f) - 1$ choices for $X_0$. For $s < r$, given $X_s$, there are $\binom{|C_{X_s}| - |X_s|}{2} = \binom{C(f) - 2s - 1}{2}$ choices for $X_{s+1}$ (this is the same as in the case of even $C(f)$). In total, the number of possible choices for $C_\alpha$ is
\[
 (C(f) - 1) \prod_{s=0}^{r-1} \binom{C(f) - 2s - 1}{2} =
 (C(f) - 1) \frac{(C(f) - 1)!}{2^r},
\]
matching the formula for $T$ in the statement of the lemma.
\end{proof}

We can now prove \Cref{thm:main-sym}.

\begin{proof}[Proof of \Cref{thm:main-sym}]
Let $\mathcal{F}$ be a $2$-intersecting subset of $S_n$ of size $(n-2)!$, where $n \geq 2$, and let $f$ be its characteristic function. Our goal is to show that
\[
 \mathcal{F} = \{ \alpha \in S_n : \alpha(i_1) = j_1 \text{ and } \alpha(i_2) = j_2 \}
\]
for some $i_1 \neq i_2 \in [n]$ and $j_1 \neq j_2 \in [n]$. In other words, we need to show that $\mathcal{F}$ is a $2$-coset.

Suppose first that $n \geq 8$. Since $n \geq 5$, \Cref{thm:MR21} shows that $\deg f \leq 2$. \Cref{lem:C-bound-sym} shows that $C(f) \leq n-2$. \Cref{lem:main-ub-sym} implies that $C(f) \geq 2$ and $|\mathcal{F}| \leq T (n - C(f))!$, where $T = \lfloor C(f)/2 \rfloor (C(f) - 1)!/2^{\lfloor C(f)/2 \rfloor}$. Therefore
\[
 \frac{(n-2)!}{(n - C(f))!} \leq \frac{2\lfloor C(f)/2 \rfloor (C(f) - 1)!}{2^{\lfloor C(f)/2 \rfloor}}.
\]
The left-hand side is clearly increasing in $n$. Since $n \geq C(f) + 2$, it follows that
\[
 \frac{C(f)!}{2} \leq \frac{2\lfloor C(f)/2 \rfloor (C(f) - 1)!}{2^{\lfloor C(f)/2 \rfloor}},
\]
and so
\[
 C(f) \leq \frac{2 \lfloor C(f)/2 \rfloor}{2^{\lfloor C(f)/2 \rfloor - 1}}.
\]
If $C(f) = 2r$ is even then this reads $2r \leq 2r/2^{r-1}$, and so $r = 1$. If $C(f) = 2r+1$ is odd then this reads $2r+1 \leq 2r/2^{r-1}$, which never holds. We conclude that $C(f) = 2$.

Let $\alpha \in \mathcal{F}$ be arbitrary. Since $C(f) = 2$, there is a certificate of the form $\{(i_1,j_1),(i_2,j_2)\}$ for $\alpha$. The number of permutations satisfying this certificate is $(n-2)!$, and we conclude that $\mathcal{F}$ consists of all permutations satisfying the certificate, completing the proof in this case.

\smallskip

In order to complete the proof, we consider the case $n \leq 7$. It suffices to show that for $n \in \{2,3,4,5,6,7\}$, any $2$-intersecting family of size $(n-2)!$ which contains the identity permutation is of the form
\[
 \{ \alpha \in S_n : \alpha(i_1) = i_1, \alpha(i_2) = i_2 \}
\]
for some $i_1 \neq i_2$.

If $n = 2$ or $n = 3$ then $(n-2)! = 1$, and the claim can be verified directly. In order to verify the remaining cases $n \in \{4,5,6,7\}$, we formulate the task as a maximum clique problem. Let $G_n$ be the graph whose vertex set consists of all permutations which $2$-intersect the identity permutation, and in which two permutations are connected if they $2$-intersect. A $2$-intersecting family containing the identity permutation is the same as a clique in $G_n$. Using SAGE~\cite{sagemath}, which internally uses the software library Cliquer~\cite{Cliquer}, we verify that a maximum clique in $G_n$ contains $(n-2)!$ many permutations, and furthermore, there are exactly $\binom{n}{2}$ many maximum cliques, matching the number of $2$-cosets containing the identity permutation. The relevant code appears in \Cref{apx:main-sym}. This completes the proof.
\end{proof}

\section{Main theorem: perfect matching scheme} \label{sec:main-pms}

The proof of \Cref{thm:main-pms} is similar to that of \Cref{thm:main-sym}, but there is an additional complication: \Cref{lem:degree1-cert-pms} only states that degree~$1$ Boolean functions have certificate complexity~$2$. This causes the proof of \Cref{lem:structure-sym} to fail. In order to get around this, we show that the bad case never occurs for \emph{(inclusion-)maximal} $2$-intersecting families.

Fix $n \geq 1$, and work over $\cM_{2n}$. We start by defining \emph{extended certificates}. An extended certificate of size $c$ is either a standard certificate of size $c$, or a pair $(C',\{i,j,k\})$, where $C'$ is a certificate of size $c-1$ and $i,j,k \in [2n]$ are distinct elements not appearing in $C'$. A perfect matching $m \in \cM_{2n}$ satisfies a certificate $(C',\{i,j,k\})$ if it satisfies $C'$ and doesn't contain any of the pairs $\{i,j\},\{i,k\},\{j,k\}$.

\begin{lemma} \label{lem:extended-certificates}
Fix $n \geq 1$, work over $\cM_{2n}$, and let $c \leq n - 2$. Let $C_1 = (C'_1,\{i,j,k\})$ be an extended certificate of size $c$, and let $C_2$ be an extended certificate of size $c$. If any perfect matching satisfying $C_1$ and any perfect matching satisfying $C_2$ are $2$-intersecting, then the same holds when $C_1$ is replaced by $C'_1$.
\end{lemma}
\begin{proof}
Suppose that $m_1$ is a perfect matching satisfying $C'_1$ but not $C_1$. Without loss of generality, $m_1(i) = j$. We need to show that if $m_2$ is a perfect matching satisfying $C_2$ then $m_1$ and $m_2$ are $2$-intersecting.

Since $m_1$ satisfies $C_1$, its pair representation consists of $C_1$ together with $n-(c-1) \geq 3$ more pairs, one of which is $\{i,j\}$. Among the other pairs, at least one $\{a,b\}$ does not involve $k$. Consider the following two perfect matchings:
\begin{align*}
    m'_1 &= m_1 \setminus \bigl\{\{i,j\}, \{a,b\}\bigr\} \cup \bigl\{ \{i,a\}, \{j,b\} \bigr\}, \\
    m''_1 &= m_1 \setminus \bigl\{\{i,j\}, \{a,b\}\bigr\} \cup \bigl\{ \{i,b\}, \{j,a\} \bigr\}.
\end{align*}
Both of these perfect matchings avoid the edges $\{i,j\},\{i,k\},\{j,k\}$. If $m_2$ contains $\{i,a\}$ then it cannot contain $\{i,b\}$ or $\{j,a\}$, and similarly for $\{j,b\},\{i,b\},\{j,a\}$. Therefore $m_2$ cannot intersect both $\bigl\{ \{i,a\}, \{j,b\} \bigr\}$ and $\bigl\{ \{i,b\}, \{j,a\} \bigr\}$. Suppose that $m_2$ does not intersect $\bigl\{ \{i,a\}, \{j,b\} \bigr\}$. Then $m'_1 \cap m_2 \subseteq m_1 \cap m_2$. On the other hand, $m'_1$ satisfies $C_1$, and so $|m'_1 \cap m_2| \geq 2$. We conclude that $|m_1 \cap m_2| \geq 2$, as needed.
\end{proof}

\subsection{Structure theorem} \label{sec:structure-pms}

Two cosets $\dcoset{i_1}{j_1},\dcoset{i_2}{j_2}$ are \emph{compatible} if their intersection is non-empty. A family is \emph{$r$-covered} if it is contained in the union of $r$ compatible cosets.

\begin{lemma} \label{lem:covering-pms}
Let $f$ be the characteristic function of a $2$-intersecting family $\mathcal{F} \subseteq \cM_{2n}$, and suppose that $C(f) \leq n-2$.

If $m \in \mathcal{F}$ then $\mathcal{F}$ is $r$-covered for $r = C(f,m) - 1$.
\end{lemma}

Since the proof is identical to that of \Cref{lem:covering-sym} (with \Cref{lem:2-int-coset-pms} standing for \Cref{lem:2-int-coset-sym}), we omit it.

The following structure theorem is the analog of \Cref{lem:structure-sym}. There are three differences in the statement. First, we assume that $\mathcal{F}$ is \emph{maximal}: it is not properly contained in any larger $2$-intersecting family. Second, we assume that $r \leq n-2$. Third, we only provide certificates to perfect matchings which belong to $\mathcal{F}$.

\begin{lemma} \label{lem:structure-pms}
Let $\mathcal{F} \subseteq \cM_{2n}$ be an $r$-covered family, say $\mathcal{F} \subseteq \dcoset{i_1}{j_1} \cup \cdots \cup \dcoset{i_r}{j_r}$, where $r \leq n-3$. Suppose that $\mathcal{F}$ is maximal, and that its characteristic function $f$ has degree at most~$2$.

Every $m \in \mathcal{F}$ has a certificate of the form $\bigl\{\{i_1,m(i_1)\},\ldots,\{i_r,m(i_r)\},\{i,j\}\bigr\}$.
\end{lemma}

Note that the certificate could mention the same pair twice, for example if $m(i_1) = i_2$.

\begin{proof}
The first step is to prove that for every distinct $k_1,\ldots,k_r$, the restriction of $f$ to $\dcoset{i_1}{k_1} \cap \cdots \cap \dcoset{i_r}{k_r}$ has degree at most~$1$, assuming the coset intersection is non-empty. This part is identical to the proof of \Cref{lem:structure-sym}, so we do not repeat it here (the only difference is that we never consider $k_1,\ldots,k_r$ such that the coset intersection is empty).

Now suppose that $m \in \mathcal{F}$. Then the restriction of $f$ to $I = \dcoset{i_1}{m(i_1)} \cap \cdots \cap \dcoset{i_r}{m(i_r)}$ (which is non-empty) has degree~$1$. According to \Cref{thm:degree1-classification-pms}, $f|_I$ is either a dictator, a triangle, or an anti-triangle. If $f|_I$ is a dictator or a triangle, then $m|_I$ has a certificate of size~$1$, and we are done. We would like to show that the remaining case is impossible.

Suppose that $f|_I$ is an anti-triangle. Then $m$ has an extended certificate $(C',\{i,j,k\})$ of size $r+1 \leq n-2$. The same argument shows that every perfect matching in $\mathcal{F}$ has an extended certificate of size $r+1 \leq n-2$. Therefore \Cref{lem:extended-certificates} implies that if $m' \notin \mathcal{F}$ is a perfect matching satisfying $C'$, then $\mathcal{F} \cup \{m'\}$ is also $2$-intersecting. Hence either $\mathcal{F}$ is not maximal (if such $m' \notin \mathcal{F}$ exists), or in fact $C'$ is already a certificate for $m$.
\end{proof}

\subsection{Proof of main theorem} \label{sec:proof-pms}

We can now prove the analog of \Cref{lem:main-ub-sym}, from which \Cref{thm:main-pms} will easily follow. The only difference in the statement is the additional assumption that $\mathcal{F}$ is maximal.

\begin{lemma} \label{lem:main-ub-pms}
Let $\mathcal{F} \subseteq \cM_{2n}$ be a maximal $2$-intersecting family whose characteristic function $f$ has degree at most~$2$, and suppose that $C(f) \leq n-2$. Let $C_1(f)$ the maximum certificate complexity of any $m \in \mathcal{F}$.

The family $\mathcal{F}$ is contained in the union of at most $T$ many $C_1(f)$-cosets, where
\[
 T = \frac{2\lfloor C_1(f)/2 \rfloor (C_1(f) - 1)!}{2^{\lfloor C_1(f)/2 \rfloor}}.
\]
Moreover, $C_1(f) \geq 2$.
\end{lemma}
\begin{proof}
Let $\mu \in \mathcal{F}$ be a perfect matching satisfying $C(f,\mu) = C_1(f)$. \Cref{lem:2-int-coset-pms}, applied to $m_1 = m_2 = \mu$, shows that $C(f,\mu) \geq 2$, and so $C_1(f) \geq 2$. \Cref{lem:covering-pms} shows that $\mathcal{F}$ is $r$-covered, where $r = C(f,\mu) - 1$. Note that $r \leq C(f) - 1 \leq n - 3$. Without loss of generality, we can assume that $\mathcal{F} \subseteq \dcoset{1}{r+1} \cup \dcoset{2}{r+2} \cup \cdots \cup \dcoset{r}{2r}$.

\Cref{lem:structure-pms} shows that every $m \in \mathcal{F}$ has a certificate of the form $\bigl\{\{1,m(1)\},\ldots,\{r,m(r)\},\{i,j\}\bigr\}$, which we call the \emph{standard certificate}. The \lcnamecref{lem:structure-pms} also implies that $\mathcal{F}$ is not $(r-1)$-covered: otherwise, every perfect matching in $\mathcal{F}$ would have a certificate of size $r < C_1(f)$, which contradicts the definition of $C_1(f)$.

A standard certificate $\bigl\{\{1,m(1)\},\ldots,\{r,m(r)\},\{i,j\}\bigr\}$ satisfies $m(1),\ldots,m(r) \notin [r]$, since otherwise \Cref{lem:structure-pms} would imply that $\mathcal{F}$ is $(r-1)$-covered (since some of the pairs $\{t,m(t)\}$ coincide), which is impossible. This allows us to color the pairs in a standard certificate unambiguously: a pair $\{t,m(t)\}$ is colored $t$, and a pair $\{i,j\}$ with $i,j \notin [r]$ is colored $r+1$. Every standard certificate contains exactly one pair of each color.

The remainder of the proof is identical to that of \Cref{lem:main-ub-sym}.
\end{proof}

We can now prove \Cref{thm:main-pms}. The proof is very similar to that of \Cref{thm:main-sym}.

\begin{proof}[Proof of \Cref{thm:main-pms}]
Let $\mathcal{F}$ be a $2$-intersecting subset of $\cM_{2n}$ of size $(2n-5)!!$, where $n \geq 2$, and let $f$ be its characteristic function. Our goal is to show that
\[
 \mathcal{F} = \{ m \in \cM_{2n} : m(i_1) = j_1 \text{ and } m(i_2) = j_2 \}
\]
for some distinct $i_1,j_1,i_2,j_2 \in [2n]$.

Suppose first that $n \geq 8$. Since $n \geq 3$, \Cref{thm:FMS} shows that $\mathcal{F}$ is maximal and that $\deg f \leq 2$. \Cref{lem:C-bound-pms} shows that $C_1(f) \leq n-2$. \Cref{lem:main-ub-pms} implies that $C_1(f) \geq 2$ and  $|\mathcal{F}| \leq T (2n - 2C_1(f) - 1)!!$, where $T = \lfloor C_1(f)/2 \rfloor (C_1(f) - 1)!/2^{\lfloor C_1(f)/2 \rfloor}$. Therefore
\[
 \frac{(2n-5)!!}{(2n - 2C_1(f) - 1)!!} \leq \frac{2\lfloor C_1(f)/2 \rfloor (C_1(f) - 1)!}{2^{\lfloor C_1(f)/2 \rfloor}}.
\]
The left-hand side is clearly increasing in $n$. Since $n \geq C_1(f) + 2$, it follows that
\[
 \frac{(2C_1(f))!}{3 \cdot 2^{C_1(f)} C_1(f)!} =
 \frac{(2C_1(f) - 1)!!}{3} \leq \frac{2\lfloor C_1(f)/2 \rfloor (C_1(f) - 1)!}{2^{\lfloor C_1(f)/2 \rfloor}},
\]
and so 
\[
 (2C_1(f))! \leq 3 \cdot 2^{\lceil C_1(f)/2 \rceil} C_1(f)! (2\lfloor C_1(f)/2 \rfloor) (C_1(f) - 1)!.
\]

If $C_1(f)$ is even then this reads
\[
 (2C_1(f))! \leq 3 \cdot 2^{C_1(f)/2} C_1(f)!^2.
\]
When $C = C_1(f)$ increases by $1$, the left-hand side increases by a factor of $(2C+2)(2C+1)$, while the right-hand side increases by a factor of $\sqrt{2} (C+1)^2$, which is smaller for $C \geq 2$. When $C = 3$, the inequality fails. Therefore in this case, $C_1(f) = 2$.

If $C_1(f)$ is odd then the inequality reads
\[
 (2C_1(f))! \leq 3 \cdot 2^{(C_1(f)+1)/2} C_1(f)! (C_1(f) - 1)(C_1(f) - 1)! < 3 \cdot 2^{(C_1(f)+1)/2} C_1(f)!^2.
\]
As before, if we increase $C = C_1(f)$ by $1$, the left-hand side increases by a larger factor than the (far) right-hand side. When $C = 3$, the inequality fails, and so this case cannot happen.

Now let $m \in \mathcal{F}$ be arbitrary. Since $C_1(f) = 2$, there is a certificate of the form $\bigl\{\{i_1,j_1\},\{i_2,j_2\}\bigr\}$ for $m$. The number of permutations satisfying this certificate is $(2n-5)!!$, and we conclude that $\mathcal{F}$ consists of all perfect matchings satisfying the certificate, completing the proof in this case.

\smallskip

In order to complete the proof, we consider the case $n \leq 7$. It suffices to show that for $n \in \{2,3,4,5,6,7\}$, any $2$-intersecting family of size $(2n-5)!!$ which contains the perfect matching $\{1,n+1\},\ldots,\{n+2n\}$ is of the form
\[
 \{ m \in \cM_{2n} : m(i_1) = n + i_1, m(i_2) = n + i_2 \}
\]
for some $i_1 \neq i_2 \in [n]$.

If $n = 2$ or $n = 3$ then $(2n-5)!! = 1$, and the claim can be verified directly. In order to verify the remaining cases $n \in \{4,5,6,7\}$, we use exhaustive search just as in the proof of \Cref{thm:main-sym}. The relevant code appears in \Cref{apx:main-pms}.
\end{proof}

\bibliographystyle{alpha}
\bibliography{biblio}

\appendix

\section{SAGE code} \label{apx:SAGE}

The code snippets below form part of the proofs of \Cref{lem:C-bound-sym,thm:main-sym,thm:main-pms}. They are also available as part of the arXiv version of the paper.

\subsection{Code for Lemma \ref{lem:C-bound-sym}} \label{apx:C-bound-sym}

\begin{lstlisting}
import itertools

def degree_2_functions():
    "Generate vector space of degree 2 functions, represented as truth table"
    return span(QQ, [vector(int((x & m) == m) for x in range(16)) \
                     for m in [0, 1, 2, 4, 8, 3, 5, 9, 6, 10, 12]])

def verify_no_sensitivity_4():
    "Verify that no function from {0,1}^4 to {0,1} has sensitivity 4 at zero"
    degree_2 = degree_2_functions()
    return not any(f for f in itertools.product([0,1], repeat=16) \
                   if f[0] != f[1] == f[2] == f[4] == f[8] and \
                   vector(f) in degree_2)
\end{lstlisting}

\subsection{Code for Theorem \ref{thm:main-sym}} \label{apx:main-sym}

\begin{lstlisting}
def intersection_size(a, b):
    "Size of intersection of two permutations"
    return sum(i==j for (i,j) in zip(a,b))

def generate_G_n(n, t):
    "Generate graph whose cliques are t-intersecting families of S_n containing id"
    idp = list(range(n))
    V = [a for a in Permutations(range(n)) if intersection_size(a, idp) >= t]
    E = [(a,b) for a in V for b in V if t <= intersection_size(a, b) < n]
    return Graph([V, E])

def verify_main_theorem_n(n):
    "Verify that all 2-intersecting families containing id are 2-cosets for given n"
    G_n = generate_G_n(n, 2)
    return G_n.clique_number() == factorial(n - 2) and \
        len(G_n.cliques_maximum()) == binomial(n, 2)

def verify_main_theorem():
    "Verify that all 2-intersecting families containing id are 2-cosets for 4 <= n <= 7"
    return all(verify_main_theorem_n(n) for n in [4,5,6,7])
\end{lstlisting}

\subsection{Code for Theorem \ref{thm:main-pms}} \label{apx:main-pms}

\begin{lstlisting}
def intersection_size(a, b):
    "Size of intersection of two perfect matchings"
    return len(a.intersection(b))

def generate_G_n(n, t):
    "Generate graph whose cliques are t-intersecting families of M_2n containing a fixed PM"
    idm = frozenset(PerfectMatchings(2*n).__iter__().__next__())
    V = [frozenset(a) for a in PerfectMatchings(2*n) if \
        intersection_size(frozenset(a), idm) >= t]
    E = [(a,b) for a in V for b in V if t <= intersection_size(a, b) < n]
    return Graph([V, E])

def verify_main_theorem_n(n):
    "Verify that all 2-intersecting families containing a fixed PM are 2-cosets for given n"
    G_n = generate_G_n(n, 2)
    return G_n.clique_number() == factorial(2*n-4)/(2^(n-2) * factorial(n-2)) and \
        len(G_n.cliques_maximum()) == binomial(n, 2)

def verify_main_theorem():
    "Verify that all 2-intersecting families containing a fixed PM are 2-cosets for 4 <= n <= 7"
    return all(verify_main_theorem_n(n) for n in [4,5,6,7])
\end{lstlisting}

\end{document}